\documentclass[onenumber, plainnumbering,english,11pt]{amsart}
\usepackage[utf8]{inputenc}
\usepackage[square,numbers]{natbib}
\usepackage{amssymb,amsmath,amscd,mathrsfs,paralist}
\usepackage{bm}
\usepackage[T1]{fontenc}
\usepackage{graphicx}
\usepackage[all,cmtip]{xy}
\newcommand{\N}{\mathbb N}
\newcommand{\Z}{\mathbb Z}
\newcommand{\C}{\mathbb C}
\newcommand{\Ccc}{\mathbb {C}}

\newcommand{\F}{{\mathbb F}}

\newcommand{\tauu}{{\mathfrak t}}
\DeclareMathOperator{\Id}{Id}
\newcommand{\nc}{\newcommand}
\nc{\BCc}{{\mathbb{C}(\wp(z),\wp^\prime(z))}}
\nc{\BC}{{\mathbb C}}
\nc{\BQ}{{\mathbb Q}}
\nc{\BR}{{\mathbb R}}
\nc{\BZ}{{\mathbb Z}}
\nc{\BP}{{\mathbb P}}
\nc{\BN}{{\mathbb N}}
\nc{\BM}{{\mathbb M}}
\nc{\fH}{{\mathfrak{H}}}
\nc{\vp}{{\varepsilon}}\nc{\dpar}{{\partial}}\nc{\al}{{\alpha}}
\nc{\PSL}{PSL(2,\BR)}
\nc{\PS}{PSL(2,\BZ)}
 \nc{\CL}{PSL(2,\BZ/m\BZ)}
 \newtheorem{numbered}{}
\newtheorem{theorem}[numbered]{Theorem}
\newtheorem{proposition}[numbered]{Proposition}
\newtheorem{definition}[numbered]{Definition}
\newtheorem{lemma}[numbered]{Lemma}

\newtheorem{example}[numbered]{Example}

\theoremstyle{remark}
\newtheorem{remark}[numbered]{Remark}
\begin{document}
\title[Anharmonic solutions and elliptic modular functions]{Anharmonic Solutions to the Riccati equation and elliptic modular functions}
\author{Ahmed Sebbar and Oumar Wone}
\address{Ahmed Sebbar and Oumar Wone\\
Univ. Bordeaux, IMB 
   \\UMR 5251, F-33400 Talence, France.}
    \email{ahmed.sebbar@math.u-bordeaux1.fr}
      \email{oumar.wone@math.u-bordeaux1.fr}
\date{}
\begin{abstract}
We study the irreducible and algebraic equation 
\begin{equation}
\label{eq0}
x^n+a_1x^{n-1}+\cdots+a_n=0,
\end{equation}
with $\,n\geqslant4$, over the differential field $(\F=\Ccc(t),\delta=^\prime)$ of rational functions over the field of complex numbers $\C$ with $\delta$ its usual derivation. We assume that a root of the equation is solution of a Riccati differential equation 
$$u^{\prime}+B_{0}+B_{1}u+B_{2}u^{2}=0$$
where $B_{0}$, $B_{1}$, $B_{2}$ are in $\F$.
 \noindent Under this condition we obtain a parameterization (depending on the Galois group) of the equations in terms of a variable $T\in \F$: one can put equation \eqref{eq0} in the form
$$F_{n,H}(x,T)=0$$
with coefficients in $\Ccc$ and with $H$ the Galois group of \eqref{eq0}. The polynomial $F_{n,H}(x,T)$ is universal in the sense that for any subgroup $H$ of $PGL(2,\Ccc)$, there is essentially one such (see the statement of theorem \ref{you}). Our motivation for studying only those Riccati equations which have algebraic solutions of degree $\geqslant4$ is that our main technique of classification uses the cross-ratio. Indeed the cross-ratio of any four solutions of a Riccati equation is constant.  
On the other hand we give an example of a degree $3$ irreducible polynomial equation satisfied by certain weight $2$ modular forms for the subgroup $\Gamma(2)= \{M \in SL(2,\Z), M= \Id\mod 2\}$ of $SL(2,\Z)$, whose roots satisfy a same Riccati equation on the differential field $\left(\C(E_2,E_4,E_6),\dfrac{d}{d\tau}\right)$ with $E_i$ the Eisenstein series of weight $i$ respectively. These solutions are related to a Darboux-Halphen system. Finally our interest is the following problem: for which "potential" $q\in\C(\wp,\wp^\prime)$ does the Riccati equation $\dfrac{dY}{dz}+Y^2=q$ admit algebraic solutions over the differential field $\C(\wp,\wp^\prime)$, $\wp$ being the classical Weirstra§ elliptic function? In this problem, our approach will be through Darboux polynomials and invariant theory. We find that one of the conditions that the minimal polynomial $\Phi(x)$ of $u$ must satisfy is $\tau_4(\Phi)(x)=0$, with $\tau_4(\Phi)(x)$ the fourth transvectant of $\Phi(x)$.
\end{abstract} 
\maketitle
\tableofcontents
 \section{Introduction}
 \label{chapric}
 Let $(\F=\Ccc(t),\delta=')$ be the usual field of rational functions over $\Ccc$ with its usual derivation. We proceed in this paper to the study of some particular properties of the Riccati equation 
 \begin{equation}
\label{eqr90}
u^{\prime}+B_{0}+B_{1}u+B_{2}u^{2}=0, \quad B_{0}, B_{1}, B_{2}\in \F.
\end{equation}
We consider an algebraic irreducible equation $h_n(x)$ of degree $n\geqslant4$ 
 \begin{equation}
 \label{eqr89}
h_n(x)=x^n+a_1x^{n-1}+\cdots+a_n=0,
\end{equation}
 where the $a_i$s belong to the field $\F$. Suppose that a root of $h_n$ is a solution of a Riccati equation. As we are in characteristic zero, every irreducible polynomial is separable and hence has a transitive Galois group $H$. So if one root of equation \eqref{eqr89} satisfies the equation \eqref{eqr90}, one sees that all its roots do because the classical and differential Galois groups coincide (cf. section \ref{one}). Therefore our hypothesis is identical to giving oneself an irreducible separable $h_n(x)$ whose roots satisfy a same Riccati equation of the form \eqref{eqr90}. 
 
 A natural question is what possible implications does the fact that all its roots satisfy \eqref{eqr90} has on the form of equation \eqref{eqr89}?

The cross-ratio of any four solutions $x_i$ of equation \eqref{eqr89} is a constant and we will call such equations, anharmonic equations \cite{pic1877}. Our choice of this terminology, anharmonic, is motivated by its link with the anharmonic Weirstrass $\wp$ function which will appear later in section \ref{drach} (equation \eqref{lame}).
The fundamental property of constance of the cross-ratio leads to interesting consequences; it is indeed this constance which will enable us to effectively find the anharmonic polynomials. In the following we will not distinguish between the $h_n(x)$ \eqref{eqr89} which differ only by the following "functorial" equivalence. Namely
\begin{equation}
\label{eqr1001}
\begin{cases}
    x &\to\frac{ax+b}{cx+d}   \text{ }, \\
     h_n(x) &\to\widetilde{h}_n(x)=\dfrac{K}{(cx+d)^n}h_n\left(\frac{ax+b}{cx+d}\right) \text{}\\
     &\text{$ac-bd\not=­0$ and $a$, $b$, $c$, $d\in\F$, $K\in\F^{\star}$}.
\end{cases}
\end{equation} 

This identification makes sense because such a transformation changes a solution of a Riccati equation into another solution of a Riccati differential equation. Moreover it preserves the cross-ratio of any four elements in some field extension of $\F$. Also it does not change the irreducible nature of $h_n(x)$. And finally it induces an isomorphism between the Galois groups of the equations, as such an action does not introduce any new algebraic relation among the roots of $\widetilde{h}_n$ that can not be deduced from the  algebraic relations satisfied by the roots of $h_n$ and vice-versa.

We will take advantage of similar substitutions  in order to put the anharmonics equations into a normal form in a sense that will be made precise in section \ref{one}.

The main result that we obtain is
\begin{theorem}
\label{you}
Let be an integer $n\in\N\,,n\geqslant4$. Let $H\subset PGL(2,\Ccc)$ be a finite subgroup.
We consider the equation
$$u^\prime+B_0+B_1u+B_2u^2=0\qquad (\star),$$
a Riccati equation having algebraic solutions of degree $n$ and of Galois group $H$.
There exists a polynomial $F_{n,H}(X,Y)$ in two variables and with constant coefficients ($\in\Ccc$) such that for every algebraic solution $u$ of ($\star$) with minimal polynomial $h_n$ considered modulo equivalence there exists $T\in\F\setminus\Ccc$:
$$h_n(x)=F_{n,H}(x,T).$$   
\end{theorem}
This paper is structured as follows: in section \ref{one} we give a brief review of differential Galois theory and prove that the Galois groups of the anharmonics are finite subgroups of $PGL(2,\Ccc)$. We finish this section with the proof of the main theorem \ref{you}.
Then we give the possible degrees of the anharmonics, in the tables at the end of subsection \ref{two}. An example is also given in subsection \ref{two}. 
The general literature focuses only on the cases where the Galois group is tetrahedral, octahedral or icosahedral. In the present work (section \ref{sric5}) we give all the cases, including the cyclic and dihedral cases. These last two cases are, remarkably enough, related to elliptic functions (Darboux-Halphen and Ramanujan systems) and modular functions. This is in our opinion one of the relevant point of this paper.
We then study the irreducible degree $3$ polynomial coming from modular forms and the Darboux-Halphen system. In section \ref{drach} we are interested in the following problem: for which elliptic function $q$ does the Riccati equation $\dfrac{dY}{dz}+Y^2=q$ admit algebraic solutions over the differential field $\C(\wp,\wp^\prime)$? This section is highly dependent on algebraic invariant theory. In particular we show that for a solution $u$ of this Riccati equation to be algebraic over $\C(\wp,\wp^\prime)$: $\Phi(u)=0$, with $\Phi(x)\in\C(\wp,\wp^\prime)[x]$, it is necessary that the fourth transvectant $\tau_4(\Phi)(x)$ of its minimal polynomial $\Phi(x)$ to vanish. This is summarized in the following theorem
\begin{theorem} For $\Phi(x)=x^n+ \frac{n}{1!}a_1x^{n-1}+\frac{n(n-1)}{2!}a_2x^{n-2}+\ldots+a_n$ with $a_1=0$ to be a minimal polynomial over the differential field $\C(\wp,\wp^\prime)$ of a solution $u$ of $\dfrac{dY}{dz}+Y^2=q$, with $q\in\C(\wp,\wp^\prime)$, it is necessary that
	 $$X(\Phi)(x)= -nx\Phi ,\quad \tau_4 (\Phi)(x)= 0$$
	 where $X=\dfrac{\partial}{\partial z}+(q-u^2)\dfrac{\partial}{\partial u}$.
	 \end{theorem}
The most important observation raised by the previous theorem is certainly the vanishing of the fourth transvectant as condition for the existence of algebraic solutions of the Riccati equation. This vanishing of the fourth transvectant is related to the generalized Chazy equation (cf. section \ref{drach}).
In section \ref{section5} we make the link between the previous study of section \ref{drach} and hypergeometric differential equation and we finish in section \ref{section6} by raising some questions.

\section{Algebraic solutions of the Riccati equation on $\Ccc(t)$ and the cross-ratio}
\label{one}
Before proving the theorem \ref{you} we recall certain needed facts about differential Galois theory from which we deduce some preliminary results. Our main reference is \cite{sing2003}.

A differential field $(\F,\delta)$ is a field together with a derivation $\delta$ on $\F$. One says that a differential field $(\mathfrak{F},\Delta)$ is a differential field extension of $(\F,\delta)$ if $\mathfrak{F}$ is a field extension of $\F$ and its derivation $\Delta$ extends $\delta$. 

In the following $y^{\prime}$ and $y^{\prime\prime}$ will stand for $\delta(y)$ and $\delta^2(y)$. Let an ordinary homogeneous second order linear differential equation
$$L(y)=y^{\prime\prime}+b_1y^{\prime}+b_2y=0, \quad b_i\in \F$$
be given. Let $\eta,\zeta$ be a fundamental set of solutions of $L$ (two independent solutions generating its two dimensional vector space of solutions $V$ over $\Ccc$). The differential extension field 
 $$\mathfrak{F}=\F<\eta,\zeta>=\F(\eta,\eta^\prime,\zeta,\zeta^\prime)$$
is called a Picard-Vessiot extension if in addition $\mathfrak{F}$ and $\F$ have the same field of constants. The differential Galois group $\mathfrak{G}(L)$ of $\mathfrak{F}$ over $\F$ is the group of differential automorphisms of $\mathfrak{F}$ that leave $\F$ invariant (an automorphism $\sigma$ is differential if $\sigma(a^\prime)=(\sigma a)^\prime$, $a\in \mathfrak{F}$).  
 
The previous choice of a basis of solution induces a faithful representation of $\mathfrak{G}(L)$ as a subgroup of $GL(2,\Ccc)$, defined in the following way: for $\sigma\in\mathfrak{G}(L)$ one has
 \begin{equation*}
\begin{cases}
    \sigma(\eta)=a_\sigma\eta+b_\sigma\zeta \\
    \sigma(\zeta)=c_\sigma\eta+d_\sigma\zeta.
\end{cases}
\end{equation*}
 A different choice of basis $(\eta_1,\zeta_1)$ leads to equivalent representations as there exists $M\in GL(2,\Ccc)$: $(\eta_1,\zeta_1)=(\eta,\zeta)M$. We identify these equivalent representations. Also one can show, (see \cite[p.19]{sing2003}), that the differential Galois group considered as a subgroup of $GL(2,\Ccc)$ is an algebraic subgroup.
We first prove that the Galois groups of the anharmonic equations $h_n(x)$ \eqref{eqr89} are finite.
\begin{proposition}
Let $h_n(x)=x^n+a_1x^{n-1}+\cdots+a_n=0,\,a_i\in\F$ be an anharmonic equation then its Galois group is finite and coincides with a finite subgroup of the differential Galois group of the associated second order linear differential equation.
\end{proposition}
\begin{proof}
One knows that the mapping $y\mapsto\dfrac{y^\prime}{y}$ defines a surjective map between the set of non trivial solutions of $L(y)=y^{\prime\prime}+b_1y^\prime+b_2y=0$ and the solutions of the Riccati equation $u^\prime+B_2u^2+B_1u+B_0=0$, with the $B_i$ functions of the $a_i$ and their derivatives. We suppose that this Riccati equation has solutions which are algebraic of degree $\geqslant4$. 


Let $x_i,\,i=1,\cdots,n$ be its roots in a decomposition field over $\F$. There exists $\dfrac{y_i^\prime}{y_i},\,i=1,\cdots,n$ such that $x_i=\dfrac{y^\prime_i}{y_i}$. Using a fundamental basis of solutions $(\eta,\zeta)$ of $y^{\prime\prime}+b_1y^\prime+b_2y=0$, one sees that the $x_i$ all belong to the Picard-Vessiot extension $\mathfrak{F}=\F<\eta,\zeta>=\F(\eta,\eta^\prime,\zeta,\zeta^\prime)$. But the field extension $\F(x_i,\,i=1,\cdots,n)|\F$ of $\F$ is a differential subfield of the Picard-Vessiot extension as $x_i^\prime+B_2x_i^2+B_1x_i+B_0=0$ for any $i$. Thus it corresponds by the Galois correspondence \cite[p.~25]{sing2003} to an algebraic subgroup $H$ of the differential Galois group $\mathfrak{G}(L)\subset GL(2,\Ccc)$. Moreover the extension $\F(x_i,\,i=1,\cdots,n)|\F$ is normal sub-extension of the Picard-Vessiot extension, as any element of the differential Galois group $\mathfrak{G}(L)$ will send a root of $h_n$ to a root of $h_n$ and also send a solution of the given Riccati equation to a solution of the same equation. Therefore $\F(x_i,\,i=1,\cdots,n)|\F$ is a Picard-Vessiot extension, $H$ is a normal in $\mathfrak{G}(L)$ and the differential Galois group of $\F(x_i,\,i=1,\cdots,n)|\F$ is exactly $\mathfrak{G}(L)/H$. At the same time we remark that any element of the differential Galois group of $\F(x_i,\,i=1,\cdots,n)|\F$ belongs to the classical Galois group of the equation $h_n$ and conversely. Indeed let $\sigma$ belong to the ordinary Galois group of $h_n$; by hypothesis it is transitive and permutes the roots $x_i$ of $h_n$. If $\sigma(x_i)=x_j$ for $i­j$ (the case $i=j$ being clear), one has
$$\sigma(x_i)^\prime=x_j^\prime=-(B_2x_j^2+B_1x_j+B_0)$$ 
and 
$$\sigma(x_i^\prime)=\sigma(-(B_2x_i^2+B_1x_i+B_0))=-B_2\sigma(x_i)^2-B_1\sigma(x_i)-B_0=-B_2x_j^2-B_1x_j-B_0;$$
 therefore extending to the field $\F(x_i,\,i=1,\cdots,n)$, one sees that any element of the classical Galois group is also differential and fixes $\F$, ie belongs to the differential Galois group of $\F(x_i,\,i=1,\cdots,n)|\F$, which is thus a finite group.
  
Finally any element of $ \mathfrak{G}(L) \bigcap \Ccc^\star$, the subgroup of those $\sigma \in \mathfrak{G}(L)$ that act on the solution space $V$ as scalar multiplication, has a trivial action on any $\dfrac{y^\prime}{y}$, for $y$ any non-trivial solution of $L(y)=y^{\prime\prime}+b_1y^\prime+b_2y=0$. Hence by passing to the quotient one gets a finite subgroup of $PGL(2,\Ccc)$.
\end{proof}

The list of the finite algebraic subgroups of $PGL(2,\Ccc)$ is provided in \cite[section2]{bald1979}. We have 
\begin{proposition}
Let G be a finite subgroup of $PGL(2,\Ccc)$. Then either
\begin{itemize}
\item G is a cyclic subgroup, or
\item G is dihedral, or
\item the order of $G$ is $12$ $($the tetrahedral case$)$, or
\item the order of $G$ is $24$ $($the octahedral case$)$, or 
\item the order of $G$ is $60$ $($the icosahedral case$)$.
\end{itemize}
\end{proposition}
These groups admit well-known representations as fractional linear transformations (loc. cit). The tetrahedral subgroup $A_4$ of $PGL(2,\Ccc)$ is isomorphic to the group of linear fractional transformations generated (under composition) by the elements
$$\Theta_2: z\to -z;\quad \epsilon_0:z\to\dfrac{1}{z};\quad \theta_1:z\to\dfrac{1-i}{1+i}\dfrac{z+1}{z-1}$$
with $i$ a solution in $\Ccc$ of $z^2+1=0$.
The dihedral subgroup $\mathfrak{D}_m$ of $PGL(2,\Ccc)$ is isomorphic with the subgroup of linear fractional transformations generated by
$$\Theta_m:z\to\xi z;\quad\epsilon_0:z\to\dfrac{1}{z},\,\text{with $\xi^m=1$ in $\Ccc$}.$$
The octahedral subgroup of $PGL(2,\Ccc)$ is on his side generated by the two linear fractional transformations
$$\theta_1: z\to\dfrac{1-i}{1+i}\dfrac{z+1}{z-1},\quad \theta_2: z\to\dfrac{1-i}{1+i}z.$$
The icosahedral subgroup of $PGL(2,\Ccc)$ comes from the linear fractional transformations
$$\Theta_5:z\to\xi z \quad \epsilon_0:z\to\dfrac{1}{z},\quad \epsilon_1:z\to\dfrac{\alpha z+\beta}{\beta z-\alpha},\,\text{with $\xi^5=1$, $\alpha=\dfrac{\xi^4-\xi}{\sqrt{5}}$ and $\beta=\dfrac{\xi^2-\xi^3}{\sqrt{5}}$ in $\Ccc$}.$$
Finally the cyclic subgroups of order $n$ of $PGL(2,\Ccc)$ are generated by
 $$\Theta_n:z\to\xi z,\, \text{with $\xi^n=1$ in $\Ccc$}.$$
To each of the previous groups of fractional linear transformations is associated an absolute invariant \cite{klein}.

For the tetrahedral case, an absolute invariant is given by the formula
$$\Psi_{tetr}(z)=\dfrac{\psi\left(z\right)}{\phi\left(z\right)}=\left(\dfrac{z^4+1+2i\sqrt{3}z^2}{z^4+1-2i\sqrt{3}z^2}\right)^3.$$
For the dihedral cases one can consider the map
$$\Psi_{dih}(z)=\dfrac{\psi\left(z\right)}{\phi\left(z\right)}=\dfrac{1}{z^m}+z^{m}.$$
Absolute invariants for the octahedral and icosahedral groups are given respectively by
$$\Psi_{oct}(z)=\dfrac{\psi\left(z\right)}{\phi\left(z\right)}=\dfrac{(1+14z^4+z^8)^3}{108z^4(1-z^4)^4}$$
and
$$\Psi_{ico}(z)=\dfrac{\psi\left(z\right)}{\phi\left(z\right)}=\dfrac{(-(z^{20}+1)+228(z^{15}-z^5)-494z^{10})^3}{1728z^5(z^{10}+11z^5-1)^5}.$$
The cyclic group admits the invariant
$$\Psi_{cyc}(z)=\dfrac{\psi\left(z\right)}{\phi\left(z\right)}=z^n.$$
The properties of the absolute invariants $\Psi$ (for any of the five given expressions above) are the following:
\begin{itemize}
  \item They are left invariant under the transformation of the mentioned subgroups, with the action given by
  $$(h_f,\Psi)\to h_f.\Psi:z\to \Psi(h_f.z).$$
  \item Every other absolute invariant for the group is a rational function with coefficients in $\Ccc$ of $\Psi$.
\end{itemize}
Before continuing we state the constance of the cross-ratio of any four solutions of a Riccati equation \cite[p.~63]{davis}.
\begin{lemma}
Let $(\F,\delta=^\prime)$ be a differential field and $u^\prime+B_2u^2+B_1u+B_0=0$ a Riccati differential equation with coefficients in $\F$. Let be  $u_1$, $u_2$, $u_3$, $u_4$ four arbitrary distinct solutions of the Riccati equation. Then the derivative of the cross-ratio
$$\dfrac{(u_1-u_3)(u_1-u_4)}{(u_2-u_3)(u_2-u_4)}$$
vanishes identically. 
\end{lemma}
\begin{proof}[Proof of Theorem 1.1]
Our main ingredient in the proof is the constance of the cross-ratio of any four given solutions of a Riccati equation. Nevertheless we are aware that other proofs could be given with more differential Galois theory machinery.

The above absolute invariants will be useful in order to express the roots of the anharmonics. We will then use the constance of the cross-ratio to give the form of the general solution of the Riccati equation. 

For each anharmonic $h_n$ with Galois group $H \subset PGL(2,\Ccc)$ (as represented above), consider one of its roots $x_i$. It is therefore solution of the associated Riccati equation $u^\prime+B_2u^2+B_1u+B_0=0$. 
Let $H_{x_i}$ be the stabilizer of $x_i$. As the Galois group of $h_n$ is transitive, one has if $N=\#H$ denotes its order, $N=np$, where $p$ is the cardinal of the stabilizer $H_{x_i}$ and $n$ is the degree of $h_n$. Consider the action of $H_{x_i}$ on the set of roots of $h_n$. By constance of the cross-ratio no element of $H$ different from the identity element can fix more than two roots. Thus two cases can happen. $H_{x_i}$ does not fix any other root $x_j­x_i$; then its action on a root $x_j­x_i$ gives rise to an orbit $H_{x_i}x_j$ of length $p$. By the disjointness of the different orbits and the fact that the total sum of their cardinals is $n-1$, one has: $p|(n-1)$. In the same way one sees that $p|(n-2)$ if it happens that $H_{x_i}$ fixes another root $x_j­x_i$. 

Hereafter we will show in the analysis of the subsection \ref{two} that all the stabilizers are cyclic. 

Consider the rational fraction
$$\Psi(Z)-T$$
where $T\in\F\setminus{\C}$ is such that the numerator of $\Psi(Z)-T$ is irreducible over $\F$.
Choose $\tauu$ as a root of its numerator (which is a function of $t$ possibly multivalued): 
\begin{equation}
\label{eqpr1}
\Psi(\tauu)=\dfrac{\psi(\tauu)}{\phi(\tauu)}=T.
\end{equation}
The extension $\C(\tauu)/\C(t)$ is algebraic of degree $N$. Moreover if $\tauu$ is chosen as such and the associated group is not cyclic, then any other root of $\Psi(Z)-T$ is in $\C(\tauu)$ (see \cite{bald1979} for this particular fact). 

Now we construct an anharmonic $h_n$ from the rational function $\Psi$. To this end we have, with the previous notations \eqref{eqpr1}  
\begin{proposition}
Let $H$ be a finite homography subgroup of $PGL(2,\C)=PSL(2,\C)$ of cardinal $N=np$. Let $U(Z)$ be a polynomial of degree $n$ with complex coefficients. Assume also that 
$$U\left(\dfrac{aZ+b}{cZ+d}\right)=\dfrac{k}{(cZ+d)^n}U(Z)\qquad(\star\star)$$
where $k$ is in $\C$ and $ \displaystyle g= \left(\begin{array}{cc}
           a  &  b \\
           c &  d\\
 \end{array}\right)\  \in H$ (such a polynomial exists by classical invariant theory, \cite[chap. 2]{lamotke}). Let $\eta_i$ be n distinct complex numbers which constitute an orbit under the action of $H$, with the $\eta_i$ distinct from the roots of $U(Z)$. Let $\Psi$ be an absolute invariant of the given finite homography subgroup of $PGL(2,\C)$. Define
$$x_i=\dfrac{\left[\dfrac{1}{\tauu-\eta_i}-\dfrac{U^\prime(\tauu)}{nU(\tauu)}\right]}{\Psi^\prime(\tauu)}:=\mathfrak{f}(\tauu,\eta_i).\qquad (\star\star\star)$$
Then the $x_i$ satisfy a same Riccati equation in the variable $\tauu$ and the equation descends to an equation over $\C(t)$.
\end{proposition}
\begin{proof}
To deduce the first part of the proof we remark that
$$\eta_i=\tauu-\left[x_i\Psi^\prime(\tauu)+\dfrac{U^\prime(\tauu)}{nU(\tauu)}\right].$$
Derivation with respect to $\tauu$ gives
\begin{equation}
\label{rico}
\Psi^\prime x_i^\prime+x_i\Psi^{\prime\prime}+\dfrac{d}{d\tauu}\dfrac{U^\prime}{nU}+\left(x_i\Psi^\prime+\dfrac{U^\prime}{nU}\right)^2=0,
\end{equation}
a Riccati equation with respect to the variable $\tauu$ satisfied by $x_i$. 
Let us show that the general solution $u$ of the previous Riccati equation admits a representation in the form $u=\mathfrak{f}(\tauu,C)$, $C$ an arbitrary constant. Indeed take three solutions $x_1$, $x_2$, $x_3$ with respective representations
$$x_1=\mathfrak{f}(\tauu,\eta_1),\quad x_2=\mathfrak{f}(\tauu,\eta_2),\quad x_3=\mathfrak{f}(\tauu,\eta_3).$$
Define $\rho$ by the implicit equation
$$u=:\mathfrak{f}(\tauu,\rho).$$
We show that $\rho$ is constant. The cross-ratio of $u$, $x_1$, $x_2$, $x_3$, is a constant $K_c$ . So
$$\dfrac{(u-x_1)(x_3-x_1)}{(x_3-x_2)(u-x_2)}=K_c.$$
But by ($\star\star\star$)
$$\dfrac{(u-x_1)(x_3-x_1)}{(x_3-x_2)(u-x_2)}=\dfrac{(\rho-\eta_1)(\eta_3-\eta_2)}{(\eta_3-\eta_1)(\rho-\eta_2)}=K_c.$$
So $\rho$ is a constant whose value depends on $K_c$. It is the arbitrary constant $C$. Now by an easy computation one sees that the expression $\mathfrak{f}(\tauu,C)$ is invariant under a simultaneous linear fractional transformation $\left(\tauu\to\dfrac{a\tauu+b}{c\tauu+d},\,\,C\to\dfrac{aC+b}{cC+d}\right)$ and $g= \left(\begin{array}{cc}
           a  &  b \\
           c &  d\\
 \end{array}\right)\  \in H$ using ($\star\star$) and ($\star\star\star$). That is $$\mathfrak{f}(\tauu,C)=\mathfrak{f}\left(\dfrac{a\tauu+b}{c\tauu+d},\dfrac{aC+b}{cC+d}\right).$$
Therefore $u=\mathfrak{f}(\tauu,C)$ is at the same time the general solution of 
 $$\dfrac{du}{dt}+T^\prime(t)\dfrac{u\Psi^{\prime\prime}+\dfrac{d}{d\tauu}\dfrac{U^\prime}{nU}+\left(u\Psi^\prime+\dfrac{U^\prime}{nU}\right)^2}{\Psi^{\prime2}}=0$$
 and of the corresponding equation with $\tauu$ replaced by $\dfrac{a\tauu+b}{c\tauu+d}$. Thus the coefficients of the Riccati equation under consideration are invariant under the action of $H$. They are hence rational functions of $T$ and $T^\prime$ and then they belong to $\C(t)$.
\end{proof}

If we eliminate $\tauu$ between the two equations
$$T=\Psi(\tauu)\quad\text{and}\quad x_i=\mathfrak{f}(\tauu,\eta_i),$$
the resulting equation leads to the desired form of the anharmonic equation of the theorem \eqref{you}. More precisely, the resulting equation is of degree $N$; also one sees that the construction does not depend of $\eta_i$. Moreover the stabilizer of $\eta_i$ fixes the two equations. Therefore the equation that we get is the $p$-th power of a polynomial $F_{n,H}(x,T)$ $($as any of its roots will have multiplicity $p$$)$, the polynomial which we are looking for (after an eventual division by the leading coefficient). This monic polynomial of degree $n$ is in $\F[x]$ and is irreducible by the specifications on the $x_i$.
\end{proof}
\subsection{Degree of the anharmonics}
\label{two}
We have seen in this section that each anharmonic $h_n$ has a Galois group $H$ of cardinal $N=np$, with $p$ the cardinal of a stabilizer of a point (the stabilizers are conjugate) and $p|(n-1)$ or $p|(n-2)$. Moreover $H$ is either cyclic, or dihedral, or tetrahedral or octahedral or icosahedral. Also given any such group $H$, all its subgroups under the previous constraints can be a candidate for the stabilizer subgroup of a root of an anharmonic $h_n$, as in the reasoning for the construction of an anharmonic, the latter were arbitrary.

We are going to use these specifications in order to determine the possible degrees $n$. 
\begin{itemize}
\item First of all one remarks that if $p=1$, $H$ can be any of the finite subgroups of $PGL(2,\Ccc)$ and $n$ assumes an arbitrary value $\geqslant4$ in the cyclic case and arbitrary even value $\geqslant4$ in the dihedral case. It takes the value $n=12$ in the tetrahedral case, the value $n=24$ in the octahedral case and the value $n=60$ in the icosahedral case.

 \item Next assume $p>1$.
 \begin{enumerate}
 \item Then $H$ can not be cyclic. Assume the contrary; then because the group $H$ acts transitively on the roots of $h_n$, the stabilizers of the roots are conjugate. The group being cyclic, hence commutative, we see that all the stabilizers are cyclic subgroups and are the same: $H_{x_i}=H_{x_j}$ for $i­j$. Therefore $H_{x_i}$ fixes all the roots of the anharmonic $h_n,\,n\geqslant4$ and is different from the identity ($p>1$). This is can not happen because of the constance of the cross-ratio (any subgroup of $H$ which fixes more than two roots, fixes all of them and $H_{x_i}$ should be the identity).
 
\item If $H$ is dihedral then $p=2$; the case $p|(n-1)$ gives an odd $n$ and $p|(n-2)$ an even $n$. To justify this statement we remark that the dihedral group of order $N=2m$ is generated by two transformations $\Theta,\,\epsilon$ with
 $$\Theta^m=\Id;\quad\epsilon^2=\Id,\quad \epsilon^{-1}\Theta\epsilon=\Theta^{-1}.$$
It contains the $2m$ transformations
$$\Theta^{l},\, l\in\{0,\cdots,m-1\}\,\,\text{and}\,\,\epsilon\Theta^r,\,r\in\{0,\cdots,m-1\}.$$
We claim that the stabilizer of a root $x_i$, $H_{x_i}$ can not be one of $\langle\Theta^l\rangle,\,l\in\{0,\cdots,m-1\}$. Our argument is based on the transitivity of the Galois group $H$, which gives that the stabilizers $H_{x_i}$ are all conjugated. Assume to the contrary that it is of that form.

If $H_{x_k}$ is of the form $\langle\Theta^l\rangle$ and is conjugated to $H_{x_j}$ with $k­j$, via some $\epsilon\Theta^d$, then a quick computation gives
$$\Theta^{-d}\epsilon^{-1}\Theta^{l}\epsilon\Theta^d=\Theta^{-d}\Theta^{-l}\Theta^d=\Theta^{-l}.$$
Thus $H_{x_j}=H_{x_k}$.

On the other hand if  $H_{x_j}$ is conjugated to $H_{x_k}$ via some power of $\Theta$ we get again $H_{x_j}=H_{x_k}$. As this is true for arbitrary $j­k$ and the anharmonic admits at least three roots, $H_{x_k}$ fixes at least three roots of $h_n$. Hence $H_{x_k}$ must be the identity by a preceding argument. This contradicts the fact that $\sharp(H_{x_k})=p>1$.

To summarize, in the dihedral case, the stabilizers are of the form $H_{x_i} =\langle\epsilon\Theta^i\rangle$ and $p=2$.  

\item We examine the tetrahedral case $N=12=np$. The tetrahedral group contains only elements of order $2$ or $3$. If $p=2$, then $n=6$. The only case which can happen also in the tetrahedral case is $p=3$ and $n=4$.
\item If the Galois group $H$ is octahedral, then $N=24=np$. The octahedral group contains elements of order at most $4$. If $p=2$, then $n=12$; $p=3$ gives $n=8$ and finally $p=4$ corresponds to the value $n=6$. We remark that the group of order $4$ in the case of degree $6$ can not be the Klein group $\Z/2\Z\times\Z/2\Z$ as it is a normal subgroup of the octahedral group so that all its conjugates are itself. As a consequence this latter group would fix all the roots of the anharmonic $h_n,\,n\geqslant4$ and this would imply that it is the identity group, by a previous remark (constance of the cross-ratio). This would be an obvious contradiction because $p>1$. So the stabilizer are again cyclic. 
\item Lastly the icosahedral has three cases; $p=2$ and $n=30$. There appears also the case $p=3$, $n=20$ and finally the case $p=5$, $n=12$. The icosahedral group contains as is well-known only substitution of order $2$, $3$ and $5$.
\end{enumerate}
\end{itemize}
These various cases according to $p=1$ or $p>1$ are summarized in the following tables 

\begin{center}\begin{tabular}{|c|c|c|}\hline Group  & n & p \\\hline Cyclic & n & 1 \\\hline Dihedral & n & 1 \\\hline Tetrahedral & 12 & 1 \\\hline Octahedral & 24 & 1 \\\hline Icosahedral & 60 & 1 \\\hline \end{tabular}\label{defaulttable1}\\\vspace{0.5cm}The case $p=1$.
\end{center}

\begin{center}
\label{defaulttable2}
\begin{tabular}{|c|c|c|}\hline Group & p & n \\\hline -- & -- & -- \\\hline Dihedral & n & 2 \\\hline Tetrahedral & \{3,2\} & \{4,6\} \\\hline Octahedral & \{2,3,4\} & \{12,8,6\} \\\hline Icosahedral & \{2,3,5\} & \{30,20,12\} \\\hline \end{tabular}\\\vspace{0.5cm}The case $p>1$.
\end{center}
\begin{example}
We give here an example. Assume we are in the cyclic case, then $p=1$, $n=N$, $\Psi_{cyc}=\mathfrak{t}^n=T$, with $T\in\F$ such that $\Psi_{cyc}(Z)-T$ is irreducible. Also set $U(Z)=Z^n-K$, with $K\in\Ccc^\star$. Choose the $\eta_i$ as the $n$ solutions of $z^n=K^\prime$ with $ K^\prime\in \C^\star$ different from $K$. One has taking $x$ to be any of the $x_i$s 
$$x=\mathfrak{f}(\mathfrak{t},\eta)=\dfrac{1}{n\mathfrak{t}^{n-1}}\left(\dfrac{1}{\mathfrak{t}-\eta}-\dfrac{\mathfrak{t}^{n-1}}{\mathfrak{t}^n-K}\right);$$
this implies
$$\dfrac{\eta}{\mathfrak{t}}=1-\left(nTx+\dfrac{T}{T-K}\right)^{-1},\, T\in\F.$$
This can be again written
$$\dfrac{\eta}{\mathfrak{t}}=x$$
as such a transformation preserves the fundamental properties of anharmonics. In conclusion we get
$$x^n=\dfrac{K^\prime}{T}.$$
\end{example}

 \section{Riccati equation and Darboux-Halphen Systems}
\label{sric5}
In this section we will consider a remarkable irreducible third order polynomial whose three roots satisfy the same Riccati equation over the field $\C(E_2,E_4,E_6)$ endowed with its natural derivation $\dfrac{d}{d\tau}$; also we make a link between these roots and a system of Darboux-Halphen type. 

Let us first introduce some basic definitions about modular forms. We recall that there is a natural way to identify modular forms with some functions defined on lattices (rank two $\Z$-submodules of $\C$) in $\BC$ having some invariance and homogeneity properties. This identification will allow the introduction of some fundamental partial differential operators, initially considered by Fricke, Klein, Halphen, Legendre,$\cdots$. We refer to \cite{serre1970,zag2004} for details. Let $\mathfrak{R}$ be the set of lattices in the $\BR$-vector space $\BC$ and $\displaystyle{\mathfrak M}= \{(\omega _1, \omega _3) \in \BC^{*2},\, \tau=\frac{ \omega _3}{\omega _1},\,\Im\tau > 0 \}$. Let us also set $\mathfrak{H}:=\{z\in\C, \Im(z)>0\}$. Then $\mathfrak R$ is identified with the quotient ${\mathfrak M}/SL(2,\Z)$. Moreover $\BC^{*}$ acts on $\mathfrak R$ and on ${\mathfrak M}$ with two more identifications
\[{\mathfrak M}/\BC^{*} \approx  {\mathfrak H},\quad {\mathfrak R}/\BC^{*} \approx {\mathfrak H}/ \PS. \]
A function $F: \mathfrak R \mapsto \BC$ is of weight $n$ if $F(\lambda \Lambda)= \lambda ^{-n}F(\Lambda)$ for all lattices $\Lambda$ and all $\lambda \in \BC^{*}$. For $(\omega _1,\omega _3)\in \mathfrak M$ and $\Lambda= \Lambda( \omega _1, \omega _3)=\Z\omega_1\oplus\Z\omega_3$, the lattice generated by $\omega _1$ and $\omega _3$, we write simply $F(\omega _1, \omega _3)= F(\Lambda)$ so that $F(\lambda \omega _1,\lambda \omega _3)=\lambda ^{-n} F(\omega _1, \omega _3)$. Moreover $F(\omega _1, \omega _3)$ is invariant under the action of $SL(2,\Z)$ on $\mathfrak M$ (as the lattice $ \Lambda( \omega _1, \omega _3)$ is). This implies that there exists a well-defined function 
$$\displaystyle  f: {\mathfrak H} \mapsto \BC,\,F(\omega _1, \omega _3)= \omega_1 ^{-n} f\left(\dfrac{ \omega _3}{\omega _1}\right).$$
The invariance of $F$ under $SL(2,\Z)$ means when $F$ is also of weight $n$ that is $ \displaystyle f(\tau)= (\gamma \tau+ \delta)^{-n} f\left(\frac{\alpha \tau + \beta}{\gamma \tau+ \delta}\right)$. For even $n$, one recovers the classical invariance property of elliptic modular forms. Moreover in this case as $I_2$ and $-I_2$ act in the same way, one can consider that it is the modular group $PSL(2,\Z)$ which acts on the functions $F$. 
We set
 $ \displaystyle{ \omega_2=   \omega_1+ \omega_3}$ and adopt the classical notations for Weierstrass elliptic functions 
 \begin{equation}
\displaystyle \wp(u)= \wp(u; \omega_1, \omega_3)= \frac{1}{u^2}+ \sum_{m ^2+ n^{2} \neq 0} 
\left( \frac{1}{( u+ 2m \omega_1+ 2n \omega_3)^2}- \frac{1}{( 2m \omega_1+ 2n \omega_3)^2}\right).
\end{equation}
Let us now define
$$\zeta(u,\omega_1,\omega_3):=\dfrac{1}{u}+ \sum_{m ^2+ n^{2} \neq 0} 
\left( \frac{1}{( u+ 2m \omega_1+ 2n \omega_3)}+ \frac{1}{( 2m \omega_1+ 2n \omega_3)}+\dfrac{u}{( 2m \omega_1+ 2n \omega_3)^2}\right)$$
and for further use consider the modular forms (see below for definition) for the group $$\Gamma(2)=\left\{g=\in SL(2,\Z): g\equiv \Id\mod 2\right\}$$
\begin{equation*}
\begin{split}
e_1:=\wp(\omega_1;2\omega_1,2\omega_3)\\
e_2:=\wp(\omega_2;2\omega_1,2\omega_3)\\
e_3:=\wp(\omega_3;2\omega_1,2\omega_3)
\end{split}
\end{equation*}
with $\omega_2=\omega_1+\omega_3$. With this definition of the $e_k$s we further introduce another set of functions which depend on the modular variable $\tau$
\begin{equation*}
\begin{split}
e_k=:\dfrac{1}{(2\omega_1)^2}e_k(\tau)\\
k=1,2,3.
\end{split}
\end{equation*}
We will use the functions $e_k(\tau)$ in the theorem \ref{theorem} below.

Moreover with previous notations one has
\begin{equation}\label{Wei}
\zeta^{\prime}= -\wp,\quad \eta_1:= \zeta (\omega_1),\quad
\eta_3:= \zeta(\omega_3).
 \end{equation}
 
On the other hand, one of the most general principle in elliptic function theory is that for any elliptic function 
 $\phi(u, \omega_1, \omega_3)$ of periods $2\omega_1, 2\omega_3$, the two functions
 \begin{equation}
 \begin{split}
 \displaystyle f(u)&:=\omega_1 \frac{\partial \phi}{\partial \omega_1 }+ \omega_3\frac{\partial \phi}{\partial \omega_3} +u \frac{\partial \phi}{\partial u }\\
 \displaystyle g(u)&:= \eta_1 \frac{\partial \phi}{\partial \omega_1 }+ \eta_3\frac{\partial \phi}{\partial \omega_3} +\zeta(u) \frac{\partial \phi}{\partial u}
 \end{split}
 \end{equation}
 are also elliptic with the same periods $2\omega_1, 2\omega_3$. We set
 $$\tau= \frac{\omega_3}{\omega_1 },\quad\Im \tau>0,\quad q= e^{2i\pi\tau}$$
  and define the Eisenstein series
  \begin{equation}
  \begin{split}
 g_2&= 60 \sum_{m^2+ n^2 \neq 0} \frac{1}{( 2m \omega_1+ 2n \omega_3)^4},\\
 g_3&= 140\sum_{m^2+ n^2 \neq 0} \frac{1}{( 2m \omega_1+ 2n \omega_3)^6}.
 \end{split}
 \end{equation}
 As a result \cite{frosti1882}
 \begin{equation}\label{Frob1}
 \begin{split}
 -2\wp= &\omega_1 \frac{\partial \wp}{\partial \omega_1 }+
 \omega_3 \frac{\partial \wp}{\partial \omega_3} +u \frac{\partial
   \wp}{\partial u }\\
-2\wp ^2+ \frac{1}{3} g_2= &\eta_1 \frac{\partial \phi}{\partial \omega_1}+ \eta_3\frac{\partial \phi}{\partial \omega_3} +\zeta(u) \frac{\partial \phi}{\partial u}.
\end{split}
 \end{equation}
The Eisenstein series $g_2$, $g_3$ are related to the normalized Eisenstein series $E_4$ and $E_6$ by the following relations
 \begin{equation}
 \begin{split}
 E_4(\tau)=& 12 \left( \frac{\omega_1}{\pi}\right)^4 g_2= 1+ 240\sum_{n=1}^{\infty} \frac{n^3 q^n}{1- q^n}\\
 E_6(\tau)=& 216 \left( \frac{\omega_1}{\pi}\right)^6 g_3= 1- 504\sum_{n=1}^{\infty} \frac{n^5 q^n}{1- q^n}
 \end{split}
 \end{equation}
 We also set 
 $$E_2=1-24\sum_{n\geqslant1}\sigma_1(n)q^n=1- 24 q - 72 q^2- 96 q^3- 168 q^4-\cdots.$$ 
 Finally we recall (\cite[p. 84]{serre1970}) that
 $$\wp^{\prime2}=4\wp^3-g_2\wp-g_3$$
 so that 
 $$\wp^{\prime\prime}=6\wp^2-\dfrac{1}{2}g_2.$$
 When $g_2=0$ we say that we are in the equianharmonic case. Then the previous differential equation for the Weirstra§ $\wp$-function reduces to
 \begin{equation}
 \label{equian}
 \wp^{\prime\prime}=6\wp^2.
 \end{equation}
 For the following definition we refer to \cite{zag2004}.
 \begin{definition} Let $\Gamma$ be a subgroup of finite index of the the modular group $SL(2, \BZ)$. A meromorphic $($holomorphic$)$ function $f:\mathfrak H \mapsto \BC$ is a weak meromorphic $($holomorphic$)$ modular function of weight $k$ with respect to $\Gamma$ if
 $$ f\left(\frac{\alpha \tau + \beta}{\gamma \tau+ \delta}\right)= (\gamma \tau+ \delta)^{n} f(\tau)$$
 for every $\displaystyle{ 
 \begin{pmatrix}
 \alpha& \beta\\
 \gamma& \delta
 \end{pmatrix}} \in SL(2,\Z)$.
 If $f$ is meromorphic $($holomorphic$)$ at the cusps of $\Gamma$, it will be called a meromorphic $($holomorphic$)$ modular form of weight $n$ with respect to $\Gamma$ (the $q$ expansion of $f$ has no negatively indexed terms, respectively only a finite number of negatively indexed terms, in the holomorphic respectively in the meromorphic case).
\end{definition}

The above Eisenstein series are modular forms of weight $4$ and $6$ respectively for the subgroup $SL(2,\BZ)$. It is a fundamental result, that will be used later, that every modular form for $SL(2,\Z)$ is uniquely expressible as a polynomial in $E_4$ and $E_6$ and the extension ring ${\BC}[E_2,E_4,E_6]$ of ${\BC}[E_4,E_6]$ is a differential ring. More precisely the following basic relations of Ramanujan hold \cite{zag2004,ramanujan1916}:
\begin{equation} \label{S1}
 \begin{split}
\frac{1}{2i\pi}\frac{d}{d\tau}E_4&= \frac{1}{3}(E_2 E_4- E_6) \\
\frac{1}{2i\pi}\frac{d}{d\tau}E_6&= \frac{1}{2}(E_2 E_6- E_4^2)\\
\frac{1}{2i\pi}\frac{d}{d\tau}E_2&= \frac{1}{12}(E_2^2-E_4).
\end{split}
\end{equation}
In other words, the field $ \left(\BC \left( E_2, E_4, E_6\right), \dfrac{1}{2i\pi}\dfrac{d}{d\tau} \right) $ is a differential field. The subfield of constants is the field of complex numbers $\BC$ (as it is embedded in the field of meromorphic functions on $\mathfrak{H}$). 

For  the particular value $u= \omega_i,\, i= 1,2,3$, the equations \eqref{Frob1} become
  \begin{equation} \label{Frob}
  \begin{split}
-2 e_i&= \omega_1 \frac{\partial e_i}{\partial \omega_1 }+
\omega_3\frac{\partial  e_i}{\partial \omega_3},\\
-2 e_i ^2+ \frac{1}{3} g_2&= \eta_1 \frac{\partial e_i}{\partial \omega_1}+ \eta_3\frac{\partial e_i}{\partial \omega_3}.
 \end{split}
\end{equation}
We will later deal with the partial differential operator
\begin{equation*}
 D_1= \eta_1 \frac{\partial }{\partial \omega_1}+ \eta_3\frac{\partial }{\partial \omega_3}.
\end{equation*}
Its importance lies in the fact that it converts ellipticity properties into differential relations  for certain modular forms \cite{frosti1882,halphen1886,jordan1896}.
 \subsection{$ e_1, e_2, e_3$ and the Darboux-Halphen system}
 
\begin{theorem}\label{theorem}
For every $\tau, \Im \tau >0$, the function $e_k(\tau),$ $k=1, 2, 3$, solves the nonlinear differential equation of Riccati type, with coefficients in $ \BC\left(E_2, E_4, E_6\right)$
$$y^\prime= \frac{i}{\pi} \Bigl(- y^2 +  \frac{\pi^2}{3} E_2 (\tau) y 
+\frac{2}{9} \pi^4 E_4(\tau) \Bigr)$$
\end{theorem}
\begin{proof}
 We have seen in equation \eqref{Frob}, that for the operator 
$D_2= -2 \eta_1 \frac{\partial}{\partial \omega_1 }- 
2 \eta_3 \frac{\partial}{\partial \omega_3 }=-2D_1$

\[D_2 e_k= 4e_k ^2- \frac{2}{3} g_2,\; \;k= 1, 2, 3.\]
We adopt the new independent variables $\displaystyle \omega_1,~ \tau= 
\frac{\omega_3 }{\omega_1}$. Consider the identities

\[  e_k(\omega_k,2\omega_1, 2\omega_3)= (2\omega_1) ^{-2}e_k (\tau),\]
                        \[  \eta_1 \omega_1= \frac{\pi^{2}}{12}E_2 (\tau),     \quad  \eta_1 \omega_3 -\eta_3\omega_1=\frac{i\pi}{2}\]
One has
\begin{equation}
\begin{split}
D_2 e_k&= -2\eta_1\left(\frac{-2}{4\omega_1^3}e_k(\tau)-\frac{\omega_3}{4\omega_1^4}e_k^\prime(\tau)\right)-2\frac{\eta_3}{4\omega_1^3}e_k^\prime(\tau)\\
&=4\left((2\omega_1)^{-2}e_k(\tau)^2\right)-\frac{2}{36}\frac{\pi^4}{\omega_1^4}E_4(\tau).
\end{split}
\end{equation}
Simplification and rearrangement using the previously given relations give the theorem.
\end{proof} 
Thus by the previous theorem one has that: for $\omega = 2i\pi \tau,e_k(\tau)$ verifies the equation
\begin{equation} \label{S2}
 \frac{d}{d\omega}\left(e_k \right)= \frac{1}{2\pi^2}
\left(- e_k^2 +  \frac{\pi^2}{3} E_2 (\tau) e_k 
+\frac{2}{9} \pi^4 E_4(\tau) \right),\;\;1\leq  k\leq 3.
\end{equation}
The system of differential equations \eqref{S1} is now necessary to show that the system (\ref{S2}) leads to a Darboux-Halphen system. Setting $\displaystyle t= \frac{4i}{\pi} \tau$ and 
\[x= \frac{\pi^2}{12}E_2,\quad y= \frac{\pi^4}{12}E_4,\quad 
z=\frac{\pi^6}{216}E_6, \]
\noindent the system \eqref{S1} becomes
\begin{equation}
 \label{S3}
\begin{split}
\frac{dx}{dt}&= \frac{1}{2}x^2- \frac{1}{24}y \\
\frac{dy}{dt}&= 2xy- 3z\\
\frac{dz}{dt}&= 3xz-\frac{1}{6}y^2 .
\end{split}
\end{equation}
With
$X_k=\frac{\pi^2}{12}E_2+ \frac{1}{4}e_k, k=1, 2, 3$, the equations \eqref{S2} 
take the form \citep{halphen1886}
\begin{equation} \label{S4}
\begin{split}
\frac{d}{dt}(X_1+ X_2)&= X_1 X_2\\
\frac{d}{dt}(X_2+ X_3)&= X_2 X_3\\
\frac{d}{dt}(X_3+ X_1)&= X_3 X_1,
\end{split}
\end{equation}
which is a Darboux-Halphen system. Moreover it is transformable into (\ref{S3}) by means of the substitutions:
\begin{equation} \label{S5}
\begin{split}
x&= \frac{1}{3}\left(X_1 +X_2+X_3\right) \\
y&=\frac{4}{3}\left( X_1^2+ X_2^2+ X_3^2-X_1X_2- X_2X_3- X_3X_1\right)\\
z&=\frac{4}{27}\left( 2X_1- X_2- X_3\right) \left( 2X_2- X_3- X_1\right)\left( 2X_3- X_1- X_2\right).
\end{split}
\end{equation}
\subsection{The degree three polynomial for the Riccati equation}
We have seen that the fundamental differential equation
\begin{equation}
\label{weirric}
\wp^{\prime2}=4\wp^3-g_2\wp-g_3
\end{equation}
with $\wp(u):=\wp(u;1,\tau)$ holds. Also as it is well-known
\begin{equation}
\wp^{\prime}\left(\dfrac{1}{2}\right)=\wp^{\prime}\left(\dfrac{\tau}{2}\right)=\wp^{\prime}\left(\dfrac{\tau+1}{2}\right)=0;
\end{equation}
therefore $e_1(\tau)$, $e_2(\tau)$, $e_3(\tau)$ solve the equation
\begin{equation}
4X^3-g_2X-g_3=4(X-e_1(\tau))(X-e_2(\tau))(X-e_3(\tau))=0.
\end{equation}
\begin{proposition}
The solutions $e_i(\tau)$ of the Riccati equation of theorem \ref{theorem} are the roots of the irreducible degree three equation $X^3-\dfrac{g_2}{4}X-\dfrac{g_3}{4}$ on the differential field $\BC \left( E_2, E_4, E_6\right)$, endowed with the derivation $\dfrac{1}{2i\pi}\dfrac{d}{d\tau}$. The Galois group of the equation is the dihedral group $S_3$.
\end{proposition}
\begin{proof}
The $e_i(\tau)$s are modular forms of weight $2$ for $\Gamma(2)$ and they are distinct as the discriminant of the non singular elliptic curve $Y^2=4X^3-g_2X-g_3$ is the discriminant of the degree three polynomial $4X^3-g_2X-g_3$ up to non zero constant factor.
Moreover we have the well known identity

$$E_2\left(\dfrac{\alpha \tau + \beta}{\gamma \tau+ \delta}\right)=\left(\gamma \tau+ \delta\right)^2E_2(\tau)+6\dfrac{\gamma}{i\pi}\left(\gamma \tau+ \delta\right)$$ 
for all $\left(\begin{array}{cc}\alpha & \beta \\\gamma & \delta\end{array}\right)\in SL(2,\Z)$  \cite{zag2004}.

Now as 
$$(e_1+e_2+e_3)=0$$
one easily sees that the $e_i$s can not all belong to $\C(E_2,E_4,E_6)$ as the functions $E_2$, $E_4$ and $E_6$ are algebraically independent over $\C$ \cite{mahler,nishioka}. We claim that the polynomial
$$X^3-\dfrac{g_2}{4}X-\dfrac{g_3}{4}$$
whose three roots are the $e_i$s is irreducible over the differential field $\C(E_2,E_4,E_6)$. We know that $g_2$ is modular of weight $4$ and $g_3$ modular of weight $6$ for $SL(2,\Z)$. To prove the irreducibility of this polynomial, It suffices to show that it is integral over $\C[E_2,E_4,E_6]$ and then use Gauss irreducibility criterion. This is equivalent to the fact that none of the $e_i$s belongs to the latter ring. To see this is true, we assume for instance it were the case that $e_1$ belonged to $\C[E_2,E_4,E_6]=\C[E_4,E_6]\oplus\C[E_4,E_6]E_2\oplus\C[E_4,E_6]E_2^2\oplus\cdots$. $e_1$ is modular of weight $2$ for $\Gamma(2)$, using the transformation rule for $E_2$ applied to the matrix $g=\left(\begin{array}{cc}1 & 0 \\2 & 1\end{array}\right)$, one sees that $E_2$ is not modular for $\Gamma(2)$. As the weights of $E_4$ and $E_6$ are $4$ and $6$, $e_1$ can not be a polynomial in them. Thus from the previous decomposition of $\C[E_2,E_4,E_6]$ one gets a contradiction. Finally the last part follows from standard Galois theory: it depends on whether the discriminant $D$ of the polynomial $X^3-\dfrac{g_2}{4}X-\dfrac{g_3}{4}$ which is $\dfrac{1}{16}(g_2^3-27g_3^2)\propto\Delta$ ($\Delta$ is the modular discriminant), is a square or not of the ground field (if it is not a square then the Galois group is $S_3$). By $SL(2,\Z)$ modularity of $\Delta$,  if it is a square of an element of the ground field $\C(E_2,E_4,E_6)$ then its square root must belong to $\C[E_2,E_4]$. This square root must be a modular form of weight $6$ for $SL(2,\Z)$, that is a constant multiple of $E_6$. This is impossible because $\Delta$ is a cusp form and $E_6$ is not.\end{proof} 
\section{Algebraic solutions of the Riccati equation on $\BCc$: invariant theory of binary forms}
\label{drach}
In this section we investigate the following question: for which "potential" $q\in\C(\wp,\wp^\prime)$ does the Riccati equation
$$\dfrac{dY}{dz}+Y^2=q$$
admit an algebraic solution over the differential field $\BCc$? We will discover that this is given by the equianharmonic Weirstrass function $\wp_0$ corresponding to the case $g_2=0$.

First of all we recall some notions from classical invariant theory \cite{olv1995}.
Let $K$ be a field of characteristic zero. 
A binary form is by definition a homogeneous polynomial
$$ f(x,y)= a_0x^n+ \binom n1 a_1x^{n-1}y+ \binom n2 a_2x^{n-2}y^2+\cdots +a_ny^n$$
with coefficients in $K$, $n$ is the degree of the form $f$. Another used notation is 
$$ f(x,y)= (a_0, a_1, \cdots a_n)(x,y)^n.$$  On the the $K$-vector space $V_n$ of binary forms of degree $n$ the group $GL(2, K)$  acts in the following way:
\begin{equation*}
       \displaystyle g= \left(\begin{array}{cc}
           a  &  b \\
           c &  d\\
 \end{array}\right)\  \in GL(2, K),\, f \in V_n,\,\,\, (gf) (x,y)= f(ax+ by, cx+ dy).
\end{equation*}
We introduce the following important differential operator acting on polynomials of four variables $x$, $y$, $x^\prime$, $y^\prime$ over $K$
$$C:= \frac{\partial ^2}{\partial x \partial y'}- \frac{\partial ^2}{\partial x' \partial y}.$$
It is known as Cayley's process. 

For two given binary forms $Q(x,y), R(x,y)$, their transvectant of degree $r$ is the polynomial
$$ \left( Q, R \right)^{r}= \sum_{i= 0}^r (-1)^i \frac{ \partial ^r Q}{ \partial ^{r-i}x \partial ^i y}
\frac{ \partial ^r R}{ \partial ^i x \partial ^{r-i} y}.$$
The transvectant of degree $r$ can be obtained from Cayley's process using the following formula
$$\left(Q, R\right)^r= C^r\{Q(x,y)\cdot R(x',y')\}_{\vert x'=x,\, y'=y}.$$
For example: $$ \left( Q, R \right)^{1}= Q_x R_y - Q_y R_x, \quad \left( Q, R \right)^{2}= Q_{xx} R_{yy}- 2 Q_{xy} R_{xy}+ Q_{yy} R_{xx}.$$
For polynomials of one variable $F$, we  consider the projective coordinate $\displaystyle p= \frac{x}{y}$ and define the new polynomial
$\displaystyle Q(x,y)= y^n F(\frac{x}{y})$. It follows that if $ F$, $G$ are general polynomials of degrees $n$, $m$ respectively, the $r^{th}$ transvectant of $F$, $G$ is, for $r\leq \min(m, n)$ given by
$$ \left( F, G \right)^{r}= \sum_{i= 0}^r  (-1)^k \binom{r}{k}\frac{(m-k)!}{(m-r)!}\,
\frac{(n-r+k)!}{(n-r)!}\,\frac{d^{r-k} F}{dp^{r-k}}\,\frac{d^{k} G}{dp^{k}}.$$
It is of degree $\leqslant m+ n- 2r$. A few examples are given by
\begin{eqnarray*}
\left( F, G \right)^{0}&= &FG,\\
 \left( F, G \right)^{1}&=& mF_p G- nFG_p, \\
\left( F, G \right)^{2}&= &m(m-1) F_{pp}G- 2(m-1) (n-1) F_p G_p +n(n-1) FG_{pp}.
\end{eqnarray*}
Also one defines the Hessian by  
$$H(F)= \frac{1}{2} \left( F, F \right)^{2}= n(n-1) \left(F F_{pp}- \frac{n-1}{n}F_p^2\right).$$
We need in the sequel the fourth transvectant which will play an important role
\begin{equation}
\begin{split}
\tau_4(F):=\dfrac{1}{2}\left( F, F \right)^{4}&= (m-3)(m-2) ( m(m-1) F^{(4)}F- 4(m-3)(m-1) F^{(3)}F'\\
&+ 3(m-3)(m-2) {F''}^2).\end{split}
\end{equation}
\noindent In particular the fourth order differential equation corresponding to the vanishing of the fourth transvectant is
\begin{equation} \label{vanishing}
m(m-1) F^{(4)}F- 4(m-3)(m-1) F^{(3)}F'+ 3(m-3)(m-2) {F''}^2= 0.
\end{equation}
\noindent Following \cite[p. 193-194]{olv1995} we remark that the equation \eqref{vanishing} can be reduced to generalized Chazy equation which is a third order differential equation. For if $R$ is given by $ -m R= \dfrac{F'}{F}$, then an easy calculation gives
\begin{equation}
\begin{split}
\displaystyle F'&= -n R F\\
F''&= \left(-nR'+ n^2 R^2\right)F\\
F^{(3)}&= \left(-nR'' + 3n^2 RR'- n^3 R^3\right) F\\
F^{(4)}&= \left(-nR^{(3)} + 3n^2 {R'}^2 + 4n^2 R R'' - 6n^3 R'R^2 + n^4 R^4\right) F.
\end{split}
\end{equation}
\begin{lemma}
With the previous notation
$\tau_4(F)=0$ becomes the generalized Chazy equation for $R$
\begin{equation}
\label{stupido}
 R^{(3)}- 12 R R'+ 18 {R'}^2= \frac{6n^2}{n-1} \left( R'- R^2\right)^2.
 \end{equation}
\end{lemma}
\noindent On the other hand the coefficients of the fourth transvectant $\tau_4(f)$  can be computed recursively from those of the binary form $f$ \cite{brioschi1877}. 
\begin{lemma}
For $f(x,y)= (a_0, a_1, a_2, \cdots,a_n)(x,y)^n$ we have 
$$\tau_4(f)= (\alpha _0,\alpha _1, \alpha _2, \cdots,\alpha _m)(x,y)^m,\, m\leqslant2(n-4)$$ 
with
\begin{equation}
        \begin{split}
        \alpha_0 &= a_0 a_4- 4 a_1 a_3+ 3a_2^2\\
	    m_1 \alpha_1&= (n-4) (a_0a_5- 3a_1a_4 + 2a_2a_3)\\
	    m_2\alpha_2&= \frac{(n-4)(n-5)}{2}(a_0 a_6- 4a_1 a_5 + 7 a_2 a_4 -4 a_3^2)+ (n-4)^2(a_1 a_5 -4 a_2a_3 +2a_3^2)\\
	    \cdots&=\cdots\\
	    \alpha_m&= a_na_{n-4} -4 a_{n-1} a_{n-3}+ 3  a_{n-2} ^2\\
	    m_{m-1} \alpha_{m-1}&=  (n-4) (a_n a_{n-5} -3 a_{n-1} a_{n-4}+ 2a_{n-2} a_{n-3}).
\end{split}
 \end{equation}
 \end{lemma}
More generally depending on the parity of $r$ we have:
$$\displaystyle m_r \alpha _r=\sum_0^{\frac{r}{2}} p_{r,s} P_{r,s};\quad m_r \alpha _r=\sum_0^{\frac{r-1}{2}} p_{r,s} P_{r,s}$$
where
\begin{equation}
  m_r= \binom mr,\quad p_{r,s}= \binom {n-4}s \binom {n-4}{r-s},
\end{equation}
and
\begin{equation}
P_{r,s}=  a_s a_{r-s+4}- 4a_{s+1} a_{r-s+3}
 + 6a_{s+2} a_{r-s+2}- 4a_{s+3} a_{r-s+1}+ a_{s+4} a_{r-s}.
\end{equation}

\noindent We now define on the ring ${\BCc}[u]$ the following differential operator which will play a mayor role in the sequel
\begin{eqnarray} \nonumber
X:{\BCc}[u]
 &\longrightarrow&
{\BCc}[u]\\
f
&\longmapsto&
\dfrac{\partial f}{ \partial z} + (q-u^2)\frac{\partial f}{ \partial u}.
\end{eqnarray}
\begin{proof}[Proof of Theorem 2]
Let $\Phi(x)\in {\BCc}[x]$ of degree $n$ be the minimal polynomial of a solution of the Riccati equation: $Y^\prime+Y^2=q$. That is
\begin{equation}
\label{minimal poly}
\Phi(u)\, =\, u^n+ \frac{n}{1!}a_1u^{n-1}+\frac{n(n-1)}{2!}a_2u^{n-2}+\ldots+a_n=0.
\end{equation}
Then taking the derivation of the previous equation with respect to $z$ and using the Riccati equation $u^\prime+u^2=q$, one sees $X(\Phi(u))=0$ (we treat the variables $z$ and $u$ as independent ones). Hence
\begin{equation}
\label{poldar}
X(\Phi) \, =\, n(a_1 - u)\Phi.
\end{equation}
\begin{definition}
Let $K$ be a field of characteristic zero, $K[X_1\cdots,X_n]$ the ring of polynomials over it and $D$ a derivation on $K[X_1,\cdots,X_n]$. An element $P$ of $K[X_1,\cdots,X_n]$ is a called a Darboux polynomial for $D$ if there exists $R\in K[X_1,\cdots,X_n]$:
$$D(P)=RP.$$
When $R\equiv 0$ $P$ is called a first integral. $R$ is called a Darboux cofactor. Two Darboux polynomials having the same cofactor give rise to a first integral.
\end{definition}
So $\Phi(u)$ is a Darboux polynomial for the above mentioned derivation, ie a non trivial polynomial eigenvector of the derivation
$$X=\dfrac{\partial}{ \partial z} + (q-u^2)\dfrac{\partial}{\partial u}.$$
\begin{remark}
It is perhaps worthwhile to observe that in general, a function u algebraic over a differential field $(L,\delta)$ of characteristic zero verifies a linear differential equation with coefficients in L. In fact if
$u^n+a_1u^{n-1}+\cdots+a_n=0, a_i\in L,\,1\leqslant i\leqslant n$ is the minimal polynomial of f over L, then necessarily
$nu^{n-1} + (n -1)a_1 u^{n-2}+\cdots+ a_{n-1}­ 0$ and $u^\prime\in L(u)$ (one sees this last by taking the derivative of $u^n+a_1u^{n-1}+\cdots+a_n=0$). Hence $u^{\prime\prime},u^{(3)},\cdots \in L(u)$. The dimension of the the $L$-vector space $L(u)$ is $n$ and therefore the vectors $u^{(n)},\,u^{(n-1)},\cdots, u$ of $L(u)$ are linearly dependent over the field $L$, that is to say u is a solution of an homogeneous linear differential equation.
\end{remark}
\begin{remark}
\label{mee}
As our base field $K=\BCc$ is a field of meromorphic functions, the existence of a non-trivial first integral for $X$, $g(u)$, is equivalent to $g(u)$ being constant along the solutions. Therefore the general solution of our Riccati equation is given in this case by a relation (true only on the solution curve) 
$$g(u)=\mathfrak{C}$$
$\mathfrak{C}$ an arbitrary constant belonging to $\C$. Moreover two first integrals for $X$ can not be functionally independent as we know from the Cauchy theorem for differential equations, that the general solution of a first order equation depend only on one arbitrary parameter. 
\end{remark}
We have the following well-known lemma
\begin{lemma}[\cite{darboux1878,weil1995}]
\label{dar}
Let $(K,\delta)$ be a differential field and let $u^\prime=b_0+b_1u+b_2u^2$ be a Riccati differential equation over $K$. A non trivial polynomial $P(u)$ in $K[u]$ is a Darboux polynomial for the derivation $\delta + (b_0+b_1u+b_2u^2)\dfrac{\partial }{ \partial u}$ if and only if all its roots are solutions of the given Riccati equation \end{lemma}
The identity (\ref{poldar}) yields for the coefficients of $\Phi(u)$ the following relations
\begin{equation}
\begin{split}
\label{me}
 (n-1)a_2 &= na_1^2 - a_1 ' -q \\
 (n-2)a_3 &= na_1 a_2 - a_2 '- 2a_1q \\
(n-3)a_4 &= na_1 a_3 - a_3 '- 3a_2q \\
  \cdots & \cdots \cdots \cdots \cdots \cdots \\
  (n-k)a_{k+1}&= na_1 a_k - a_k '- ka_{k-1}q.
\end{split}
\end{equation}
This shows at least that if it happens that $ a_1, q$  belong to some
differential subring $L\subset K$ then  so do the other coefficients $a_k$.\\
\noindent We consider the homogeneous polynomial
\begin{eqnarray*}
K(u_1, u_2) = u_2^n \Phi(\frac{u_1}{u_2}) = \sum _{k=0}^n C_n^k a_k u_1^{n-k} u_2^k
\end{eqnarray*}
and its Hessian
 \begin{eqnarray*}
         \left|\begin{array}{cc}
            \displaystyle  \frac{\partial^2 K}{ \partial u_1^2}  &  \displaystyle  \frac{\partial^2 K}{ \partial u_1 \partial u_2} \\
	   \,                                   &     \,                                           \\
         \displaystyle  \frac{\partial^2 K}{ \partial u_1 \partial u_2} &  \displaystyle  \frac{\partial^2 K}{ \partial u_2 ^2}  \\
          \end{array}\right|\  = (n-1)u_{2} ^{2n-4} \left( n\Phi \Phi''-(n-1)\Phi'{ ^2} \right).
	  \end{eqnarray*}
	  The polynomial $H= n\Phi \Phi ''- (n-1){\Phi'}^2 \,\,\, ('=
          \frac{\partial}{ \partial u})$ is of degree $\leqslant2(n-2)$. For latter use we have the following important property
\begin{lemma}
Let $K$ be field of characteristic zero and $\Phi(u)=u^n+ \frac{n}{1!}a_1u^{n-1}+\frac{n(n-1)}{2!}a_2u^{n-2}+\ldots+a_n$ an irreducible polynomial then $H(\Phi)\not\equiv 0$. 
\end{lemma} 
\begin{proof}
This is because a polynomial with vanishing Hessian factorizes into a product of linear terms, \cite[prop. 5.3, p. 71]{kung}
\end{proof}          
Moreover the partial differential operator $X=\dfrac{\partial f}{ \partial z} + (q-u^2)\dfrac{\partial f}{ \partial u}$ is a derivation 
	  \begin{equation*}
	  X(fg) \, =\, fX(g) + gX(f),\, f,\,g\,\in\C(\wp,\wp^\prime)[u]
	  \end{equation*}
with the following commutation relation 
	   \begin{equation*}
	   [ X, \frac{\partial}{ \partial u} ] \, =\, 2u \frac{\partial}{ \partial u}.
	   \end{equation*}
Hence 
	   \begin{align*}
\begin{split}
	   X(\Phi') &= -n \Phi + (na_1-nu+2u)\Phi' \\
	   X(\Phi'') &= -2(n-1) \Phi' + (na_1-nu+4u)\Phi'' \\
	   X(\Phi \Phi'') &= -2(n-1) \Phi \Phi' + (na_1-nu+2u)\Phi \Phi'' \\
	   X({\Phi'}^2) &= 2\Phi' \left(-n\Phi + (na_1-nu+2u)\Phi' \right).
\end{split}
	   \end{align*} 
and
	   \begin{equation}\label{3'}
	   X(H) = 2 (na_1-nu+2u)H.
	   \end{equation}
The Jacobian of $\Phi$ and $H$ is
	   \begin{eqnarray*}
         \left|\begin{array}{cc}
           \dfrac{\partial \Phi}{ \partial z}  &  \dfrac{\partial \Phi}{ \partial u}  \\
	   \,                                   &     \,\\
           \dfrac{\partial H}{ \partial z} &   \dfrac{\partial H}{ \partial u} \\
          \end{array}\right|\  =  
         \left|\begin{array}{cc}
           X(\Phi)  &  \Phi_u^{'} \\
            X(H) &   H_u^{'} \\
          \end{array}\right|\ = na_1( \Phi H^{'}- 2H \Phi^{'})-u\Omega.
	  \end{eqnarray*}
with
	 $\Omega \, =\,  n\Phi H'- 2(n-2)H \Phi' .$
In a similar way we obtain
	   \begin{equation}
	   \label{omeg}
	   X(\Omega) = 3(na_1-nu+2u)\Omega
	   \end{equation}
and finally
	   \begin{equation*}
	  3 X(H)\Omega=2HX\Omega
	   \end{equation*}
 Or
	  \begin{equation}\label{3}
	   \frac{\Omega^2}{H^3} \, = \Gamma
	  \end{equation}
where $\Gamma$ is a constant. 
Taking the Jacobian of $\Phi$ and $\Omega$ we get 
	   \begin{eqnarray*}
      \displaystyle \left|\begin{array}{cc}
           \dfrac{\partial \Phi}{ \partial z}  &  \dfrac{\partial \Phi}{ \partial u}  \\
	   \,                                 &        \,                            \\
           \dfrac{\partial \Omega}{ \partial z} &   \dfrac{\partial \Omega}{ \partial u} \\
          \end{array}\right|\  =  
         \left|\begin{array}{cc}
           X(\Phi)  &  \Phi_u ' \\
            X(\Omega) &   \Omega_u ' \\
          \end{array}\right|\ = na_1( \Phi \Omega'- 3\Omega \Phi')-u\Omega_1.
	  \end{eqnarray*}
 with  $\Omega_1 \, =\,  n\Phi \Omega'- 3(n-2)\Omega \Phi' $ .
This gives	 
         \begin{equation} \label{6'}
	   X(\Omega_1) = 4(na_1-nu+2u)\Omega_1.
	   \end{equation}
The two identities  \eqref{3'}, \eqref{6'} give
	    \begin{equation*}
	   2\Omega_1X(H) = X(\Omega_1) H
	   \end{equation*}
that is with another constant $a$
	   \begin{equation*}
	   \frac{\Omega_1}{H^2} \, =\, a.
	  \end{equation*}
 As we said before in remark \ref{mee} two non trivial first integrals are functionally dependent and this should give differential relation for our $\Phi(u)$. Indeed one has the following
 \begin{eqnarray*}
	  H &=& n\Phi \Phi''- (n-1){\Phi'} ^2 \\
	   H' &=& n\Phi \Phi'''- (n-2)\Phi'  \Phi''\\
	   \Omega  &=& n \Phi H' -2(n-2) H \Phi' \\
	       &=& n^2\Phi^2 \Phi''' - 3n(n-2)\Phi \Phi' \Phi''+2(n-1)(n-2) {\Phi'}^3 \\
	       \Omega'  &=& -n(n-6)\Phi \Phi' \Phi'''+3(n-2)^2 {\Phi'}^2 \Phi''- 3n(n-2)\Phi {\Phi''} ^2 + n^2 \Phi^2 \Phi^{''''} \\
	       \Omega_1 &=& -4n^2(n-3)\Phi^2 \Phi^{'} \Phi^{'''}-3n^2(n-2)\Phi^2 {\Phi''}^2 +n^3\Phi^3 \Phi^{''''}  + 12n(n-2)^2\Phi {\Phi'}^2 \Phi''\\
	       \,&-& 6(n-1) (n-2)^2{\Phi'}^4\\
	     aH^2 &=& a\left(n^2{\Phi}^2{\Phi^{''}}^2 -2n(n-1)\Phi {\Phi^{'}}^2 \Phi^{''} +(n-1)^2{\Phi^{'}}^4  \right).  
	  \end{eqnarray*}
In conclusion  splitting the two members of the equality $\Omega_1=aH^2$ into monomials involving only powers of $\Phi^\prime$ on one side and the rest on the other side, one sees that $\Phi$ should divide $\Phi^\prime$ unless $a(n-1)^2= -6(n-1)(n-2)^2$ or $\displaystyle a= -6\frac{(n-2)^2}{n- 1}$. Moreover the divisibility of $\Omega_1-aH^2$ by $\Phi^2$ gives the vanishing of the fourth transvectant 
\begin{equation}
\begin{split}	
\displaystyle \tau_4 (\Phi)&= \displaystyle \frac{n-1}{n^2} \frac{\Omega_1 -
  aH^2}{\Phi^2}\\ 
 &= n(n-1)\Phi \Phi^{(4)}-4(n-1)(n-3)\Phi^{'}
 \Phi^{'''}+3(n-2)(n-3){\Phi^{''}}^2\\ 
& = 0.
\end{split}
	 \end{equation} 

 \end{proof}
 \begin{example}
 To give an example we look at the case of $n= 4$. In this case \[ a_2= -\frac{1}{3}q,\quad a_3= \frac{1}{6} q', \quad a_4= \displaystyle -\frac{1}{6}q''+ q^2,\]
 so that the potential $q$ solves an equation: $q'' =  8q^2$ similar to
 \begin{equation}\label{lame}
 u'' +2c u^2= 0.
 \end{equation}
  It has (see \eqref{equian}) for $ c\neq 0$, the solution
\begin{equation*}
u(z)= -\frac{3}{c}\wp _0 (z)
 \end{equation*}
 where $\wp_0 (z)$ is the equianharmonic Weierstrass function
\begin{equation*} 
 z= \int_\infty ^{\wp_0}\dfrac{dx}{2 \sqrt{x^3- 1}}.
 \end{equation*}
 \end{example}
For general  $n= 4, 6, 12$, the potential $q$ solves the equation $\displaystyle q''= 6a q^2$, with $ a=  \displaystyle\frac{(n-2)^2}{n-1}$. 

From lemma \ref{dar} we know that a Darboux polynomial for $X$ has necessarily all its roots satisfying the associated Riccati equation. As the degree of $H$ is $\leqslant2(n-2)$, we have to examine the cases $n=2$ and $n=3$ separately as in these two cases, the degree of $H(u)$ is possibly less than $n$. 
\begin{lemma}
There is no polynomial $\Phi$ of degree $2$.
\end{lemma}
\begin{proof}
Suppose that $n=2$. Then 
$$\Phi(u)=u^2+2a_1u+a_2.$$	 
One sees that $H(u)=4a_2$ using the condition $a_1=0$. But $H(u)$ is a Darboux polynomial for $X$, so \eqref{3'} $a_2^\prime=\dfrac{\partial}{\partial z}a_2=0$. This forces $a_2$ to be an element of $\C$. As $a_1=0$, $a_2\in\C$, we get a contradiction as $\Phi(u)$ is assumed to be irreducible and $\C$ is algebraically closed. Therefore the case $n=2$ for $\Phi(u)$ monic irreducible is impossible.
\end{proof}	 
\begin{lemma}
In the case $n=3$ $a_2$ does not vanish.
\end{lemma}
\begin{proof}
Indeed assume $n=3$. Then $\Phi(u)$ takes the form:
$$\Phi(u)=u^3+3a_1u^2+3a_2u+a_3$$
and $H(u)$ is $18(a_2u^2+a_3u-a_2^2)$. The previously analyzed case $n=2$ shows that if $a_2\not=0$ then $H(u)$ necessarily factors over $\BCc$. Moreover $a_2$ can not be zero, otherwise $H(u)$ becomes a polynomial of degree $0$ or $1$ and then we know that $0$ is a solution of our Riccati equation ($a_3$ can not be $0$ as $\Phi(u)$ is assumed to be irreducible) which therefore takes the form $\dfrac{du}{dz}+u^2=0$; the general solution of the latter equation is given by $u(z)=\dfrac{1}{z+C_0}\,, C_0\in\C$. We claim that such an $u$ is not algebraic over $\BCc$. Indeed if this were the case, then one would get a relation of  the form 
$$\dfrac{1}{(z+C_0)^3}+a_3=0$$
with $a_3­\not=0$ doubly periodic (minimal polynomial of degree $3$). This is impossible. Thus $a_2$ is not zero.
\end{proof} 
For the case $n=3$, the identity $\tau_4(\Phi)=0$ disappears and only the differential system \eqref{me}, with $a_1=0$ remains
 \begin{equation}
\begin{split}
 2a_2 &= -q \\
 (n-2)a_3 &= - a_2 '\\
0&= -a_3^\prime- 3a_2q
\end{split}
\end{equation}
which gives $q^{\prime\prime}=3q^2$ which is also of the above form with $a=\dfrac{(3-2)^2}{(3-1)}=\dfrac{1}{2}$.	 
\begin{remark}
The appearance of the potential $q,\,q''= 6a q^2 $ or its solution, the anharmonic Weirstrass $\wp$-function denoted $\wp_0$, is one the unifying point of this study.  
\end{remark}
\section{Link with Hypergeometric Functions and Conclusion}
\label{section5}
We consider the precedent differential equation for the potential $q$
\begin{equation}\label{11}
q^{''} = 6a q,\, a= \frac{(n-2)^2}{n-1}
\end{equation}
with the solution $q(z)= \dfrac{1}{a} \wp_0(z)= \dfrac{1}{a}\wp(z;\, g_2=0,\, g_3)$ and we would like to make a connection with some known facts on the algebraic hypergeometric functions \cite{hille1967}. For commodity reason  we set $g_3=4\lambda^6$ and the Riccati equation for $u$ becomes  
\begin{equation}\label{12}
\frac{dY}{dz}+ Y^2 = q(z) =\frac{1}{a}\wp_0(z).
\end{equation}
By making the change of variable $X= \wp_0(z)$ we obtain
\begin{equation}\label{13}
 \left(\frac{dX}{dz}\right)^2= 4X^3- g_3 = 4(X^3- \lambda^6)
\end{equation}
and if we set
\begin{equation}\label{14}
X= \lambda^2 \xi,\,\,\,\lambda z= t,\,\,\,Y= \lambda w
\end{equation}
 the equation \eqref{12} becomes
\begin{equation}\label{15}
\frac{dw}{dt}+ w^2= \frac{\xi}{a},\,\,\, \xi= \xi(t)
\end{equation}
and the equation  \eqref{13}  transforms into
\begin{equation}\label{16}
\displaystyle \left(\frac{d\xi}{dt}\right)^2= 4(\xi ^3- 1),\quad 2dt= \frac{d\xi}{\sqrt{\xi^3-1}}.
\end{equation}
\noindent Now in the equation \eqref{15} we make the change of function and variable
\begin{equation}\label{17}
t= \theta(\xi),\, \, \, v= w\theta^{'}- \frac{\theta ^{''}}{2 \theta^{ '}}
\end{equation}
which transforms a Riccati equation
\begin{equation}\label{18}
\frac{dw}{dt}+ w^2=  R(t)
\end{equation}
into another Riccati equation 
\begin{equation}\label{19}
\frac{dv}{dt}+ v^2=  r(\xi)= {\theta^{'}}^2 R(\theta) - \frac{1}{2}\{\theta, x \},
\theta^{'}= \frac{dt}{d\xi}
\end{equation}
so that \eqref{19} is a kind of Darboux-Backl{\"u}nd transformation for Riccati equations. We thus obtain a simpler form for the 
equation \eqref{12} and
\eqref{15}
\begin{equation}\label{20}
\frac{dv}{d\xi}+ v^2=  \frac{\xi}{16} \frac{ c_0 \xi ^3 -4c_1}{(\xi ^3- 1)^2}
\end{equation}
with
$$c_0 = \frac{4}{a}- 3,\, c_1 = \frac{1}{a}+ 6.$$
One can show that this last equation reduces to the desired form
\begin{equation}\label{21}
\frac{dW}{ds}+ W^2= Q(s)= \frac{1}{4}\left(\frac{\lambda ^2- 1}{s^2}+
\frac{\nu ^2 -1}{(1- s)^2}+ \frac{\lambda ^2 - \mu ^2 + \nu ^2-
1}{s(1-s)}\right)
\end{equation}
with
\begin{equation*}
\lambda= \frac{1}{3},\,\,\, \mu= \frac{n}{6(n- 2)},\,\,\, \nu= \frac{1}{2}.
\end{equation*}
\noindent Before stating the conclusion, we would like to insist on the major role of the hypergeometric function in this study. We use classical notations and recall essential facts of  the theory of Fuchs and Schwarz. The most general form of the hypergeometric equation is
\begin{equation}\label{22}
y^{''}+ \frac{(2-\lambda- \mu)x+ \lambda- 1 }{x(1- x)}y'+ \frac{(1- \lambda-
\mu)^2- \nu^2 }{4x(1- x)}y= 0.
\end{equation}
We introduce the new constants $a, b, c$ connected to the local exponents and their inverses by
$$\begin{array}{lclcl}
a+b+c= 1- \lambda- \mu,\quad a-b= \nu,\quad c = 1-\lambda ,\\
\mu_0= 1-c,\quad \mu_1=  c- a- b,\quad \mu_{\infty}= b-a\\
k_0= \dfrac{1}{\mu_0}, \quad k_1= \dfrac{1}{\mu_1},\quad k_{\infty}= \dfrac{1}{\mu_{\infty}}.
\end{array}
$$
 The associated Schwarzian equation is
\begin{equation}
-4\{s, x\}= \frac{1- \mu_0^2}{x^2(1-x)}+ \frac{1- \mu_1^2}{x(1-x)^2}-
\frac{1- \mu_{\infty}^2}{x(1-x)}.
\end{equation}
If $(k_0, k_1, k_{\infty})$ is one of the triplets $(2,2,m), \,(2,3,3),\,(2,3,4),\,(2,3,5)$, the monodromy group of the hypergeometric equation is of finite order $N= 2m$ (dihedral), $N=  12$ (tetrahedral), $N= 24$ (octahedral), $N= 60$ (icosahedral), with
$$ \frac{2}{N}= \frac{1}{k_0} + \frac{1}{k_1}+ \frac{1}{k_{\infty}}-1.$$ 
The main conclusion is that the cases for which the considered Riccati equation has algebraic solutions correspond exactly to those cases for which the hypergeometric equations has algebraic solutions. These  correspond to platonic solids or regular polyhedra.

This is a general fact, indeed \cite{bald1979,bald1980}, following Klein, have shown that any linear second order differential equation with regular singularities on an algebraic curve with a full set of algebraic solutions, is the pull-back via a rational map of a particular set of hypergeometric differential equations, called the basic Schwarz list. We refer to the mentioned references for more details. 

In our context the fact that the degree of the minimal polynomial of the found Riccati equations is $\geqslant3$ implies that each of the Riccati equations has only algebraic solutions. Thus the associated second equation has a finite projective monodromy, because given two linearly independent solutions $y_1,\,y_2$ of the associated second order equation, if we let $u_1:=\dfrac{y^\prime_1}{y_2}$ respectively $\displaystyle u_2:=\frac{y^\prime_2}{y_2}$ and denote the general solution of the Riccati by $u$, then $\dfrac{y_1^2}{y_2^2}=\propto\dfrac{(u-u_2)^2}{(u-u_1)^2}$ (see \cite[p. 30]{painl1895}). Hence the ratio of two independent solution of $y^{\prime\prime}+qy=0$ is algebraic and this is known to be equivalent to the finiteness of the projective monodromy group \cite{bald1979,bald1980}.

 We finish by a  remark but first we recall well-known facts on modular curves \cite{edwards2004}: the modular function $J: \mathfrak{H} \mapsto \mathbf{P}^1$ gives a quotient map with respect to the projective group $\PSL$. It ramifies above the points $0, 1$ with ramification indices $3$ and $2$ respectively. The group $$\Gamma(m)= \{M \in SL(2,\Z),\,M=  \Id\mod m\}$$ is a normal subgroup of $SL(2,\Z)$ and has no elliptic elements for $m\geqslant2$. If $Y(m)$ denotes the quotient of $\mathfrak H$ by $\Gamma(m)$, the cover $ {\mathfrak H} \mapsto Y(m)$ is unramified and the modular function $J$ factors over $Y(m)$, $J: Y(m) \mapsto \BC$. The modular curve $X(m)$ is the completion of $Y(m)$ obtained by adding to $Y(m)$ the finitely many cusps so that $J: X(m) \mapsto \mathbf{P}^1$. This map ramifies above  $\infty$ with order $m$. Hence the ramifications indices  above $0, 1, \infty$ are $3,2,m$ respectively. The covering group, preserving and acting transitively on the fibres, is $\CL$. When $m= 3, 4, 5$ we recover the tetrahedral, octahedral and icosahedral coverings.

In the previous theorem, the vanishing of the fourth transvectant appeared as one of the conditions for a  solution of a Riccati equation to be algebraic. In the analysis of this vanishing, the next theorem is of great importance and it is due to Brioschi and Wedekind.
\begin{proposition}[\cite{brioschi1877,edwards2004}]
\label{iff}
 \noindent Let $K$ be an algebraically closed field of characteristic zero. The fourth transvectant $\tau_{4}(f)$ of a non-trivial binary form $f$ of order $k \geq 4$ is identically zero if and only if $f$ is $GL(2, K)$-equivalent to one of the following forms
\begin{enumerate}
\item $x_1^k$ or $x_1^{k-1}x_2$ \rm{(degenerate case)}
\item $x_2(x_1 ^3 + x_2 ^3)$ \rm{(tetrahedral case)}
 \item $x_1 x_{2} \left(x_1 ^{4}+ x_2 ^4\right)$ \rm{(octahedral case)}
 \item $x_1 x_2 \left(x_1 ^{10}- 11 x_1 ^5 x_2 ^5 - x_2 ^{10} \right)$ \rm{(icosahedral case)}.
\end{enumerate}
\end{proposition}
The main point here is that the vanishing of the fourth transvectant of a non-degenerate binary form $f$ forces it to be one of the Klein forms. 
\section{Remarks and Questions}
\label{section6}
There are various way one could possibly extend the study realized in this paper. First we may consider the case of homogeneous linear third order or $n$-th order
$$y^{(n)}+a_1y^{(n-1)}+\cdots+a_ny=0,$$
$n>2$ over $(\mathbb{C}(z),\delta)$, $\delta(z)=1$. Using the substitution $u=\dfrac{y^\prime}{y}$ one reduces the homogeneous linear differential equation to a non-linear differential equation of order $n-1$. Does there exist anharmonics in this case? Moreover does the Darboux formalism in the section \ref{drach} of this paper generalize to the case of Riccati equations on the field of functions of an algebraic curve over $\C$ (the case we studied is essentially the case of elliptic curves)? Finally one possible further generalization of the Darboux formalism may occur through the study of the general Abel differential equation
$$\dfrac{dy}{dz}=y^3+q$$
with $q$ in some differential function field. 
 
\end{document}